\documentclass[11pt]{amsart}
\usepackage{palatino}
\usepackage{amsmath, amsthm}
\usepackage{amssymb, graphicx, epsfig}

\usepackage[margin=1.5cm,
  vmargin={0pt,1cm},
  nohead]{geometry}
\newtheorem{theorem}{$\quad$Theorem}[section]

\newtheorem{lemma}[theorem]{$\quad$Lemma}

\newtheorem{definition}[theorem]{$\quad$Definition}

\newcommand{\R}{{\mathbb R}}
\newcommand{\C}{{\mathbb C}}
\newcommand{\Z}{{\mathbb Z}}

\newtheorem*{theoremstar}{$\quad$Theorem}
\theoremstyle{definition}
\newtheorem{remark}[theorem]{$\quad$Remark}

\newenvironment{example}{{\bf Example. }}{}

\newcounter{bibno}

\begin{document}
\title{Encomplexed Brown Invariant of Real Algebraic Surfaces in $\R P^3$}
\author{Johan Bj\"orklund}
\date{2011}

\begin{abstract}We construct an invariant of parametrized generic real algebraic surfaces in $\R P^3$ which generalizes the Brown invariant of immersed surfaces from smooth topology. The invariant is constructed using the self intersection, which is a real algebraic curve with points of three local characters: the intersection of two real sheets, the intersection of two complex conjugate sheets or a Whitney umbrella. In \cite{KIME} the Brown invariant was expressed through a self linking number of the self intersection. We extend the definition of this self linking number to the case of parametrized generic real algebraic surfaces.
\end{abstract}

\maketitle

\clearpage

\section{Introduction}

Following the philosophy of Viro in \cite{VIRO2} we are interested in encomplexing invariants from smooth topology to construct invariants in real algebraic geometry. In \cite{VIRO2} the self linking number of a curve was encomplexed. Another example, immersions of curves in the plane have been extensively studied. Two immersions of curves lie in the same component of the space of immersions if and only if they have the same Whitney index. The Whitney index can be calculated from the self intersections in the case of a generic immersed curve. It turns out that this notion survives to (parametrized) real algebraic curves of Type I, where a corresponding encomplexed Whitney index can be calculated from self intersections (both solitary and non solitary) as proved by Viro \cite{VIRO}. 

In this paper we study a similar situation concerning the space of generically immersed oriented surfaces in $\R^3$. The Brown invariant is an invariant up to regular homotopy of immersed surfaces in $\R^3$. The Brown invariant of an immersed surface can be defined using the self intersection of the surface as has been shown by Kirby and Melvin \cite{KIME}. In their article they express the Brown invariant by constructing a curve called the ``pushoff`` close to the self intersection with a natural projection to the self intersection with $4$ points in the preimage of each point in the self intersection. The linking number between this pushoff and the self intersection is shown to give the Brown invariant.

In this paper we will encomplex the Brown invariant, using the interpretation in \cite{KIME} as a self linking number of the self intersection.
Let $M_S$ denote the space of real algebraic mappings from some smooth projective real algebraic surface $S$ into $\R P^3$. We also define two discriminants, $\sigma$ and $\gamma$. The discriminant $\sigma$ consists of those points in $M_S$ such that the corresponding parametrized surface in $\R P^3$ has topologically unstable singularities. The discriminant $\gamma$ consists consists of those points in $M_S$ such that the corresponding parametrized surface has points in its self intersection where the corresponding quadratic form in the the normal bundle of the self intersection curve has a matrix which is a nonzero multiple of the identity matrix. The discriminant $\gamma$ is dependent on the metric chosen for $\R P^3, \C P^3$. We construct an invariant called the fourfold pushoff invariant, defined on points in of $M_S\setminus(\sigma\cup\gamma)$, in section \ref{4foldpush}.

\begin{theorem}
The fourfold pushoff invariant is constant on connected components of $M_S\setminus (\sigma\cup\gamma)$.
\end{theorem}
We also show that in the case of the real algebraic surface being an immersed surface without solitary self intersections the Brown invariant coincides with the fourfold pushoff invariant in Remark \ref{coincide}.
\begin{theorem}
Counted mod $8$, the fourfold pushoff invariant is constant on connected components of $M_S\setminus\gamma$.
\end{theorem}
The corresponding smooth situation, immersed surfaces up to regular homotopy, has been studied by Goryuonov in \cite{GORY} where he describes the space of Vasiliev invariants for this situation.

For proofs see Section \ref{4foldpush}.

\section{Preliminaries}\label{prel}

Let $S$ be a smooth projective real algebraic surface together with its complexification $\C S$. We consider the space $M_S^d$ of real algebraic maps of $S$ into $\R P^3$ of degree $d$. We also equip $\R P^3$ and $\C P^3$ with two Riemannian metrics. The metrics are the metric inherited from $S^3$ and the Fubini-Study metric respectively. In Subsection \ref{mdproj} we examine the singularities of the parametrized surfaces closer. In Subsection \ref{gammasub} we examine the self intersections of the parametrized surfaces and show how to use the quadratic forms in the normal bundle of such a self intersection curve to construct the pushoff, a curve lying very close to the self intersection with a fourfold projection to the self intersection.

\subsection{Singularities in the space $M_S^d$.}\label{mdproj}
\begin{lemma}\label{projection}
The space $M_S^d$ can is diffeomorphic to the complement $C_S^d$ of a co-dimension 2 subspace of the space of projections from $P^m$ to $P^3$ for some $m$.
\end{lemma}
\begin{proof}
Since $S$ is a projective variety, we naturally have that $S\subset P^n$ for some $n$. The ring of regular functions of degree $d$ on $S$ is then generated by the monomials of degree $d$ on $P^n$.  Consider the Veronese embedding $v_d$ from $P^n$ into $P^m$. It is clear that any map $f$ of degree $d$ from $S$ to $P^3$ can be uniquely factorized as $f=\pi_h\circ v_d$ where $\pi_h$ is the projection to some 3-space $h$, since they can be represented as linear combinations of monomials from the ring of regular functions on $P^n$ in each coordinate. 
\end{proof}
\begin{remark}
Locally, the space of projections from $P^m$ to $P^3$ is naturally diffeomorphic to the projections from $A^m$ to $A^3$. Thus, a generic path in $M_S^d$ can by a compactness argument be reduced to pieces which can be considered to be paths in space of projections from $A^m$ to $A^3$, that is, a paths in $G(3,m)$.
\end{remark}
\begin{remark}
All projections do not correspond to points on $M_S^d$  since we cannot let the $m-4$-plane we project from intersect $S\subset P^m$. However, since $S$ is of dimension $2$, this is at most a co-dimension $2$ condition, so any generic path of projections can be assumed to not intersect this subvariety. 
\end{remark}

\begin{theorem}\label{codim}
The space $M_S^d$ contains a subvariety $\sigma$ of co-dimension $1$ such that outside of this co-dimension $1$ hypersurface, the only possible topological local configurations of self-intersections in the image of $f\in M_S^d\setminus\sigma$ are:
\begin{itemize}
 \item Two real sheets intersecting transversally
 \item Three real sheets intersecting pairwise transversally
 \item Two complex conjugate sheets intersecting transversally
 \item Two complex conjugate and one real sheet intersecting pairwise transversally
 \item A Whitney umbrella
\end{itemize}
\end{theorem}

\begin{proof}
We begin by noting that the allowed types of intersections are topologically stable, and so cannot be avoided. Following Theorem \ref{projection} we can study the space of projections from some $P^m$ to $P^3$. We examine the affine projections from $K^m$ to $K^3$ which form the Grassmannian $G(3,m)$ having dimension $3m-9$. The hypersurface $\sigma$ consists of the subvariety of projections which result in deeper singularities than the mentioned stable singularities. 
We will calculate the the dimension of $\sigma$ by examining the different kinds of singularities that appears in its top strata and examining their dimensions separately. Note that the stable cases mentioned concerned up to  3 distinct points ending up at the same point under the projection, such that their tangent planes were in general position, and the Whitney umbrella, arising from a projection which destroyed the tangent plane for a single point.
The higher singularities we need to consider lie in the closure of the following sets in $M_S^d$: projections with quadruple points, projections with a triple point such that the common intersection of two of the tangentplanes lie tangent to the third, projections with a double point such that the tangentplanes are tangent to each other, projections with a double point such that one point lacks a tangent plane, and one has a tangentplane (Whitney-Umbrella+plane) and projections with a single singularity where two Whitney umbrellas collide, arising from projecting along a tangentvector with an odd intersection number with the surface (under a generic projection). By Lemma \ref{Qlemma} to Lemma \ref{BKlemma} below, we know that these singularities have codimension at least $1$. Any more complicated singularity will necessarily have even higher codimension. It is also clear that this collection gives all the codimension one singularities, since any singularity must necessarily have some point in the preimage and we have examined the different singularities for $1-4$ points (and any singularity involving more than 4 points or deeper tangency would have a higher codimension than the one constructed from $4$ points).
\end{proof}
We need to understand the higher singularities in $M_S^d$. The strategy for the examinations of the codimensions of these higher singularities is examine the space of projections admitting at least one such singularity, by first calculating the dimension of the space of points chosen to end up at the singularity, then examining how many vectors are prescribed to project along and finally how to complete these prescribed vectors to a projection, at the end comparing the dimension to $3m-9$. To shorten the proof we omit the detailed calculations concerning genericity of points (since we are on a Veronese embedding in general) and the defining equations for the subvariety except for the case of two points with coinciding tangent planes where we demonstrate these calculations as well. In some of the situations we should also examine the case of two of the sheets being involved being complex conjugate sheets from the complexification of $S$. However, this does not change dimensions since choosing one point in the complexification is a $4$ dimensional choice from the real point of view, while the other point is by necessity the complex conjugate, that is, the same dimension as choosing $2$ real points. We use Goryunovs notation of the singularities from \cite{GORY}.\\
\begin{lemma}\label{Qlemma}
The points of $M_S^d$ corresponding to singularities of type Q, that is, singularities with a quadruple point are of codimension at least $1$. 
\end{lemma}
\begin{proof}
Consider the variety of projections with at least one quadruple point (regardless of if they are arising from real or complex parts of the surface). For the moment we assume that the points are real.\\ For a quadruple point in the image, we need 4 points on the surface in $P^m$ mapped to the same point by a projection. The surface is two dimensional, so the configuration space of four points on the surface have dimension $8$. The surface $S$ intersect a given m-2-hyperplane in a discrete number of points generically, so choosing points in a line/plane would decrease the dimension by at least $6,4$ respectively. To choose a projection from $P^m$ to $P^3$ we need to complete the three dimensional space generated by the four points to a $m-3$ dimensional space, i.e. we need to choose a point from $G(m-6,m-3)$ which has dimension $3m-18$. If the points chosen were situated on the same line/plane we would only be proscribed to use $1$ respectively $2$ vectors giving a increase in dimension. However, the original choice of of points had a lower dimension, compensating for this increase. The dimension of the space of projections with at least one quadruple point is then at most $3m-10$ and so the space of projections admitting a quadruple point has co-dimension 1 in the space of projections. The dimension count for quadruple points with 2 and 4 points arising from complex parts of the surface proceeds in the same manner. \\
\end{proof}
\begin{lemma}\label{Tlemma}
The points of $M_S^d$ corresponding to singularities of type T, that is, singularities with a triple points such that the three sheets intersect with a common tangent line are of co-dimension at least $1$. 
\end{lemma}
\begin{proof}
The choice of three points is $6$-dimensional. To give the three tangent planes a common line after the projection we need to choose one line from each tangent plane which we want to be the common line after the projection. This choice is $3$ dimensional. We need to choose a plane from the space spanned by these three lines to project along. This is a $2$ dimensional choice. We now have a $4-$space which we need to project along (2 dimensions for ensuring that the three points ended up at the same points, and 2 to ensure that the tangent planes had a common line). We complete it by choosing a point from $G(m-7,m-4)$ which has dimension $3m-21$. The total dimension is then $3m-10$.\\
\end{proof}
\begin{lemma}\label{HElemma}
The points of $M_S^d$ corresponding to singularities of type H and E, that is, singularities with a double points such that the two sheets are tangent (with hyperbolic tangency for type H and elliptic tangency for type E) are of co-dimension at least $1$. 
\end{lemma}
\begin{proof}
To see that projections which result in two sheets tangent to each other form a variety we observe that (in affine coordinates with the point of tangency being origo) the leading term of the polynomial defining such a degenerate surface must be $(ax+by+cz)^2+O(\bar x^3)$ where $\bar x^3$ stands for terms of order $3$. The corresponding matrix to the quadratic form is thus degenerate and of rank $1$. This can be expressed in terms of the $2\times2$ minorants of the matrix being zero. This defines a subvariety of $G(3,m)\times S$ which after projection to $G(3,m)$ gives the projections which result in these singularities.
The dimension of the space of projections with two points with having the same tangent plane in the image will now be calculated. We assume that the points are real. Two distinct points in the Veronese embedding $v_d$ of $P^n$ do not have a tangent in common as long as $d>1$, and thus especially the points on our surface do not have a tangent in common. In the case of $d=1$ and $S\subset P^n, n>3$ it is easy to see that there are only a finite number of pairs of points with common tangent plane. Thus, the choice of such points would contribute zero dimensions and completing the projection would result in dimension $3m-12$ resulting in co-dimension 3. We can ignore such points henceforth and assume that the tangent planes of the points chosen have no nonzero vectors in common. Choosing the two points contributes 4 dimensions. Ensuring that they end up at the same point under the projections gives one vector along which we need to project. The tangent planes will form a 4-dimensional linear subspace from which we need to choose a 2 dimensional subspace to project along to give a common tangent plane. Since $\dim(G(2,4))=4$ we get 4 additional dimensions. Again we wish to complete the 3 vectors by choosing a point from $G(m-6,m-3)$. Giving a total dimension of at most $3m-10$. Either the tangency is elliptic (case E) or hyperbolic (case H) (described by $x^2+y^2=\epsilon$ and $xy=\epsilon$ respectively in the common tangent plane). \\
\end{proof}
\begin{lemma}\label{Clemma}
The points of $M_S^d$ corresponding to singularities of type C, that is, singularities with a double point such that a Whitney umbrella intersects a sheet, are of co-dimension at least $1$. 
\end{lemma}
\begin{proof}
For the case of point and Whitney umbrella, we have a $4$ dimensional choice of points. To ensure that we get a Whitney umbrella, we need to destroy the tangent plane of one of the points. We need to choose one vector in the tangent plane of one of the points to project along. This is a $1$ dimensional choice. This vector, together with the vector along which we project to ensure that the two points end up at the same space gives us two vectors. To complete them we need to choose a point from $G(m-5,m-2)$ which has dimension $3m-15$. Again we have a total dimension of $3m-10$.\\
\end{proof}
\begin{lemma}\label{BKlemma}
The points of $M_S^d$ corresponding to singularities of type B,K, that is, singularities with a single point such that two Whitney umbrellas collide, are of co-dimension at least $1$. 
\end{lemma}
\begin{proof}
For the final case, of a single point, we needed to destroy the tangent plane to the point and so we have to choose one vector in the tangent plane. This tangent vector is by definition tangent to the surface. The surface has some curvature in this point. The surface only has zero curvature in finitely many points (choosing such a point would decrease the dimension by $2$, leading to co-dimension $2$). For a point with nonzero curvature only finitely many tangent lines intersect the surface with an odd intersection number (two in the case of negative curvature, zero in the case of positive curvature). Projecting along such a special tangent line will yield a different picture, namely the result of two umbrellas colliding (along either a real or a solitary self intersection). Choosing our point is a $2$ dimensional choice, and choosing our tangent line is of dimension $0$ since the choice was discrete. Completing to a projection necessitates a choice of point from $G(m-4,m-1)$ which is of dimension $3m-12$. The dimension of projections resulting in at least one point with deeper singularity than the umbrella is then of dimension $3m-10$.
\end{proof}

\begin{remark}
For a closer description of these singularities (and others of higher co-dimension) see Hobbs and Kirk \cite{KIRK}, more specifically table 1. There $B,K$ correspond to $S_k^\pm$, $H,E$ correspond to $A_0^2|A_k^\pm$, $C$ correspond to $(A_0S_0)_k$, $T$ to $A_0^3|A_k$ and $Q$ to $A_0^4$.
\end{remark}

\subsection{Constructing the pushoff from the self intersection}\label{gammasub}

Around generic points, the self intersection $C_S$ of a parametrized surface associated to a point in $M_S^d\setminus\sigma$ arises from either two real sheets intersecting along a real line or two complex conjugate sheets intersecting in a real line. A piece of the self intersection arising from two complex conjugate sheets is called a \emph{solitary} self intersection. Around isolated points the self intersection can also be either a Whitney umbrella separating a real line appearing from two complex conjugate sheets from a real line appearing from two real sheets or three real sheets intersecting, giving rise to three intersecting real lines or two complex conjugate planes and a real plane, resulting in a single real line (and two complex conjugate lines which are ignored). The set of self intersection points $C_S$ will then be a collection of immersed circles. Let the immersed circles $C_{S_i}$ be indexed by $i$. Since the circles are immersed, each circle has an associated normal bundle in $\R P^3$ while their complexification has an associated normal bundle in $\C P^3$. Given a point $p$ in $C_{S_i}$ we can examine the defining polynomial  of the surface in $\C P^3$ and examine its restriction to the normal plane $p_n$. This will associate a quadratic form to $p$ by taking the terms of at most order $2$ from the polynomial (we do not have any linear terms since we are in the self intersection). This assigns a continuous family of quadratic forms to $C_S$. We can then associate a continuous family of eigenvectors as long as the matrix associated to the quadratic form has two different eigenvalues.  For those points in $C_{S_i}$ which are not triple points or Whitney umbrellas we can assume that no such quadratic form has only one eigenvalue of multiplicity two (i.e. the associated matrix is a multiple of the identity matrix). The condition of having only one eigenvalue is of co-dimension $1$ in $M_S^d$ as described by the following lemma.

\begin{lemma}\label{gamma}
The points in $M_S^d\setminus\sigma$ which correspond to parametrizations which has a self intersection assigned a corresponding 2-form in the normal bundle which has just one eigenvalue (i.e. it is a multiple of the identity matrix) are of co-dimension 1.
\end{lemma}
\begin{proof}
Since the matrix depends on the metric, we can disturb the parametrization by some linear transformation of $\R P^3$ which is close to identity to gain a a different parametrization which do not have a corresponding matrix which is a multiple of the identity along the self intersection. Since the set of points in the image with such associated matrices is a discrete set, it is enough to show this for one point. We assume that the associated 2-form is $x^2+tx^2+y^2$ in local coordinates around the point at the self intersection $x=y=0$. By letting $L(x,y)=(x,y+\epsilon x)$ we disturb the intersection enough to remove the singular case.
The matrices which are a multiple of the identity matrix have co-dimension 2 in the space of real symmetric matrices. The self intersection itself form a one dimensional space, and so the singular maps cannot be of higher co-dimension than 1.
\end{proof}
\begin{definition}
We let $\gamma$ denote the parametrizations in $M_S$ which has points in the corresponding self intersection such that the matrix associated to the quadratic form is a multiple of the identity matrix. 
\end{definition}
\begin{remark}
Note that changing the metric on the space would change $\gamma$. In this article we assumed that the metric is the natural Fubini-Study metric on $\C P^3$. 
\end{remark}

Associated to each such form are are then four (real) eigenvectors of unit length. After scaling them by some small $\epsilon$ these four vectors will be four nonzero sections in the normal bundle of the immersed circle and thus form a fiber bundle $B_{S_i}$ over it.
We now push off our original curve as follows. Take our original immersed circle $C_{S_i}$ and include the bundle $B_{S_i}$ wherever it is defined (that is, over points which comes from exactly two sheets intersecting transversally). This will locally give us four curves lying at the boundary of an $\epsilon$ tubular neighborhood of $C_{S_i}$. At the special points where the self intersection does not look like two sheets intersecting we may extend our construction as shown by the following lemma.
\begin{lemma}
The fourfold pushoff can be extended in a continuous way to non generic points along an immersed circle $C_{S_i}$.
\end{lemma}
\begin{proof}
To see that this can be defined in a natural way we examine each situation, when necessary expressed in local coordinates $x,y,z$.
\begin{itemize}
\item Two complex conjugate sheets and one real sheet intersecting in a triple point: we ignore the real sheet and extend by continuity.
\item A Whitney umbrella, defined by $y^2-zx^2=0$: the eigenvectors are kept, one eigenvalue shifts signs.
\item Three real sheets intersecting at a triple point: Defined by $xyz=0$. The pushoff (following one of the intersection curves) will clearly be extendable by continuity, furthermore the pushoff and the self intersections will be disjoint.
\end{itemize}
\end{proof}
\begin{definition}
Given such an immersed circle $C_{S_i}$ the associated extended fourfold pushoff will be denoted by $P(C_{S_i})$ and called the \emph{pushoff} of $C_{S_i}$.
\end{definition}

\begin{lemma}
The pushoff $P(C_{S_i})$ has either one, two or four components. 
\end{lemma}
\begin{proof}
Obvious, see Section \ref{examples} for some examples.
\end{proof}
\begin{lemma}
If the immersed surface $S$ is orientable and if generic points of $C_{S_i}$ come from the intersection of two real sheets then $P(C_{S_i})$ will consist of four connected components.
\end{lemma}
\begin{proof}Choose an orientation for the surface. At each generic point of this real self intersection the corresponding normal vectors will distinguish one eigenvector. Since this distinction does not depend on any local choice, all four components are separate.  \end{proof}

\section{The Fourfold Pushoff Invariant}\label{4foldpush}
In Subsection \ref{folddef} we define the the fourfold pushoff invariant and show that it coincides with the Brown invariant for generic immersed surfaces without solitary self intersection. In Subsection \ref{11proof} we restate and prove Theorem 1.1 from the introduction. In Subsection \ref{12proof} we restate and prove Theorem 1.2 from the introduction.

\subsection{Defining the Fourfold Pushoff Invariant}\label{folddef}
We wish to define a self linking number of an immersed circle $C_{S_i}$ in the self intersection using its pushoff $P(C_{S_i}$. In general, the linking number between two chains are not well defined unless one of the chains is zero homologous. While $C_{S_i}$ is not necessarily zero homologous, we know that $H_1(\R P^3)=\Z_2$ and we know that the pushoff $P(C_{S_i})$ is four times $C_{S_i}$ in the homology (since we have a natural fourfold projections to $C_{S_i}$  and thus zero homologous. Thus, the linking number $L(P(C_{S_i},C_{S_i})$ is well defined.

\begin{definition}
We define the \emph{fourfold pushoff linking number}, $T(C_{S_i})$ of an immersed circle $C_{S_i}$ in the self intersection as the linking number between $C_{S_i}$ and its pushoff $P(C_{S_i})$. 
\end{definition}

\begin{definition}
We define the \emph{fourfold pushoff invariant} $T(S')$ of surface $S'$ corresponding to a point  $p\in M_S\setminus (\sigma\cup\gamma)$ as $\sum_{i\in I} T(C_{S_i})$ where $S_i$ are immersed circles indexed by some set $I$ is an immersed circle in the self intersection of $S$.
\end{definition}
\begin{remark}
Note that the fourfold pushoff invariant takes the value zero on embeddings.
\end{remark}
\begin{remark}\label{coincide}
In the affine situation with no solitary self intersections the fourfold pushoff coincides with Kirby and Melvins definition of Browns invariant in \cite{KIME}.
 While Kirby and Melvins pushoff is defined by using the sheets intersected with the boundary of a tubular neighborhood of the self intersection, it is easy to construct an isotopy taking their pushoff to $P(C_{S_i}$ as illustrated by Figure \ref{kirmel}, depicting the normal plane, by simply rotating, using the fact that $\R P^3$ is orientable.

\begin{figure}
\includegraphics[scale=0.5]{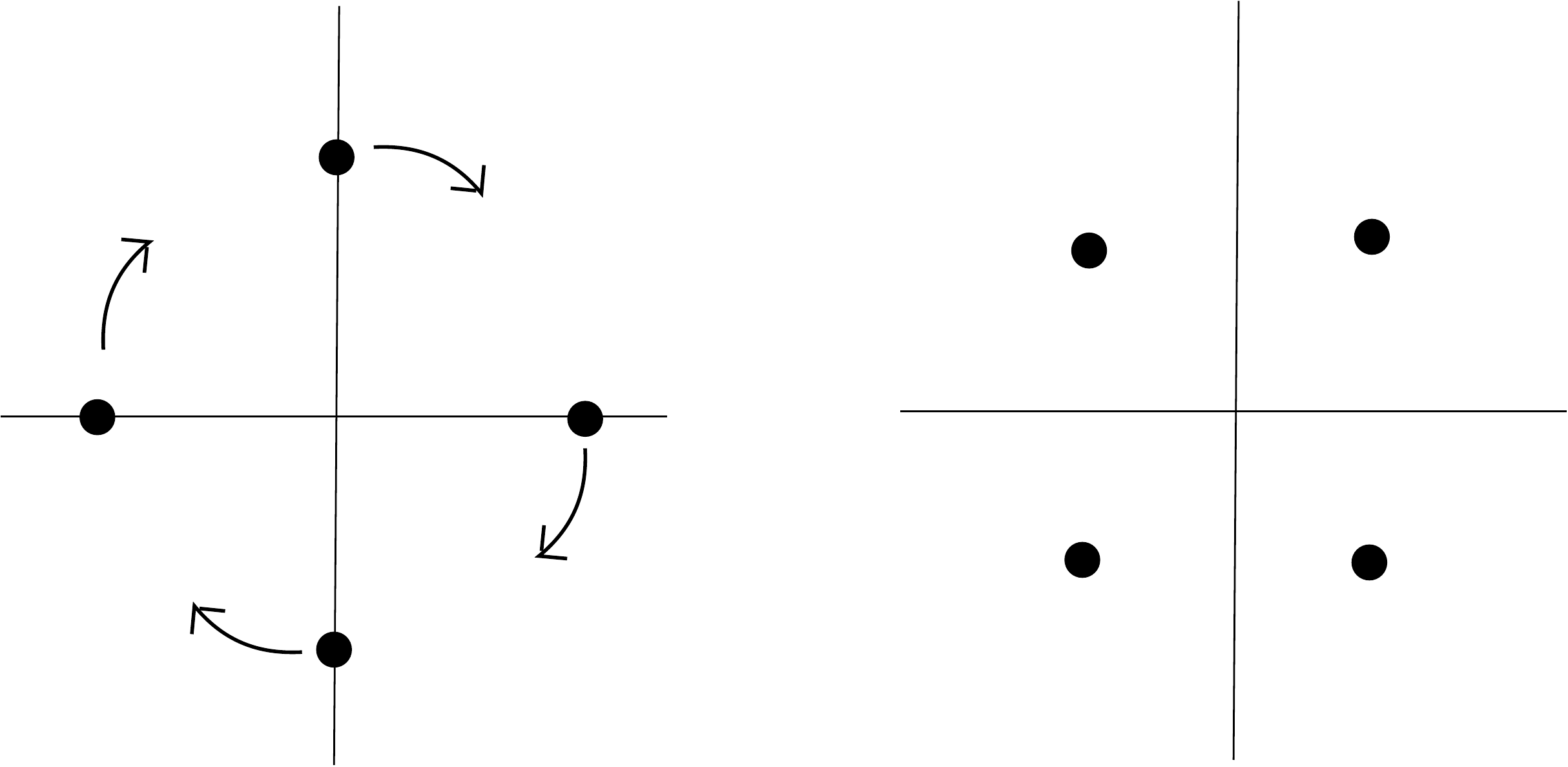}
\caption{On the lefthand side the pushoff is constructed in the normal plane using Kirby and Melvins definition, on the righthand side the pushoff is constructed using the definitions in this paper. The arrows indicate an isotopy taking the pushoffs to each other.}
\label{kirmel}
\end{figure} 
\end{remark}
\begin{remark}\label{metricdependant}
The fourfold pushoff invariant depends on the Riemann metric chosen. To see this we could change the metric around a solitary self intersection. The eigenvectors correspond the the axis vectors of the ellipses obtained by putting the quadratic form equal to some small $\epsilon$. By changing the metric locally we can rotate these ellipses, changing the value of the invariant. As the metric would change, so would $\gamma$. The value of the invariant would change just as the chosen parametrization passes through $\gamma$.
\end{remark}

\begin{theorem}
The fourfold pushoff invariant changes by $2$ when passing transversally through $\gamma$. 
\end{theorem}
\begin{proof}
This can easily be seen by an examination of equation $zx^2+y^2+2\epsilon xy=$, modeling the crossing, where the sign of $\epsilon$ marks which side of the strata we are situated and $z=1,x=y=0$ is the point on the self intersection on which the singularity appears.
\end{proof}

The fourfold pushoff invariant changes by $2$ exactly when passing through $\gamma$. This can obviously only happen when two complex sheets intersect since the eigenvalues have different signs when two real sheets intersect.

\subsection{Proof of Theorem 1.1}\label{11proof}

We restate Theorem 1.1 presented in the introduction.
\begin{theoremstar}[1.1]
The fourfold pushoff invariant is constant on connected components of $M_S\setminus (\sigma\cup\gamma)$.
\end{theoremstar}
\begin{proof}
Obvious, since the pushoff varies continuously and is well defined on each component.
\end{proof}

\subsection{Proof of Theorem 1.2}\label{12proof}

We restate Theorem 1.2 presented in the introduction.

\begin{theoremstar}
Counted mod $8$, the fourfold pushoff invariant is constant on connected components of $M_S\setminus\gamma$.
\end{theoremstar}
\begin{proof}
From Theorem \ref{codim} and Lemma \ref{gamma} we know that $\sigma$ and $\gamma$ are of co-dimension $1$ in $M_S$. Given a path $P$ in a component of $M_S\setminus\gamma$ we can assume that $P$ passes through $\sigma$ transversally. It is then enough to show that the invariant does not change mod $8$ during such a passage. We examine these singular cases and then computing the invariant before and after the singularity. The co-dimension $1$ components of this discriminant are known from Theorem \ref{codim}. Each case can have several subcases depending on if the preimage under the projection is real or not. We recall the different situations. If we have four separate points in the preimage we we will locally have a generic intersection of four planes in a single point. This situation is denoted by $Q,Q'$ and $Q''$ depending on how many of these planes are real. If we have three distinct points, one of them is special and the other two come from two transversal sheets. These situations are denoted by $T$ and $T'$, again depending on if the two transversal sheets are complex conjugate or not. If we have two separate points, we either have a plane traveling through a Whitney umbrella which we denote by $C$, or two tangentially intersecting planes which we denote by $E,E'$ and $H,H'$ depending on if the tangency is hyperbolic or elliptic. If we have one point, we have two Whitney umbrellas colliding, either along the real self-intersection or the solitary self intersection. These are denoted by $B,K$.
For each of these subcomponents of $\sigma$ we examine the change of the fourfold pushoff.
\begin{itemize}
\item Four sheets intersecting at a single point, all of them real. This occurs when a real sheet passes through a triple point. Examining the self intersections, we see that their positions relative each other either does not change, or that one passes through another. If the two curves in the self intersection pass through each other, the invariant changes by $8$ or $0$ depending of if the curves are pieces of the same immersed circle or not. See Figure \ref{qsings}.

\begin{figure}
\includegraphics[scale=0.7]{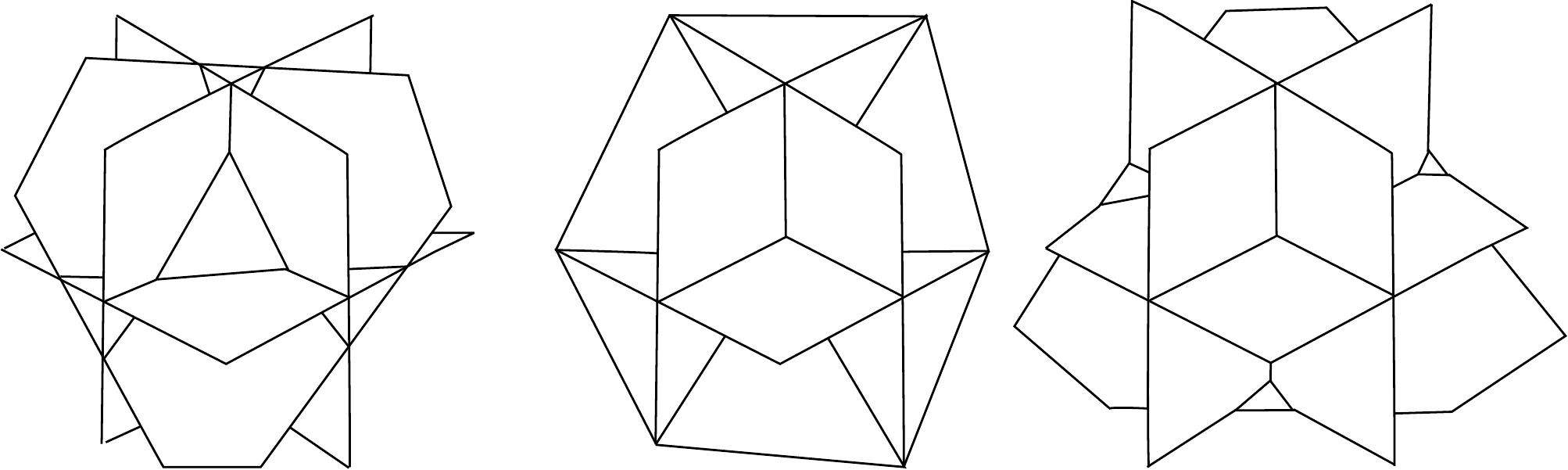}
\caption{A type $Q$ singularity}
\label{qsings}
\end{figure}

\item Four sheets intersecting, two real and two complex conjugate. This occurs when a solitary piece of the self intersection passes through a non solitary piece. The invariant changes by $8$ or $0$ depending of if the curves are pieces of the same immersed circle or not.
\item Four sheets intersecting, two pairs of complex conjugate sheets. Here two solitary pieces of the self intersection passes through each other, the invariant changes by $8$ or $0$ depending of if the curves are pieces of the same immersed circle or not.
\item Two triple points meet and annihilate each other along a real self intersection. See Figure \ref{tsings}. The invariant does not change.

\begin{figure}
\includegraphics[scale=0.7]{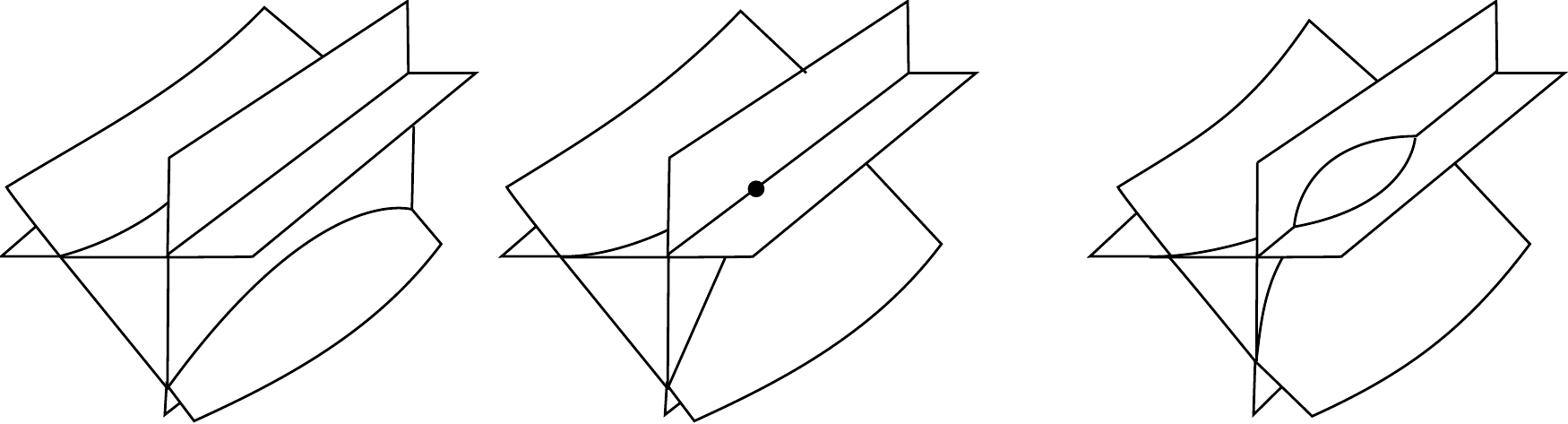}
\caption{A type $T$ singularity}
\label{tsings}
\end{figure} 

\item A solitary self intersection passes through a real plane, creating two triple points. The invariant does not change. 
\item Two real sheets intersect tangentially in a single point. See Figure \ref{esings}. The value of the invariant on the new circle is $0$.

\begin{figure}
\includegraphics[scale=0.7]{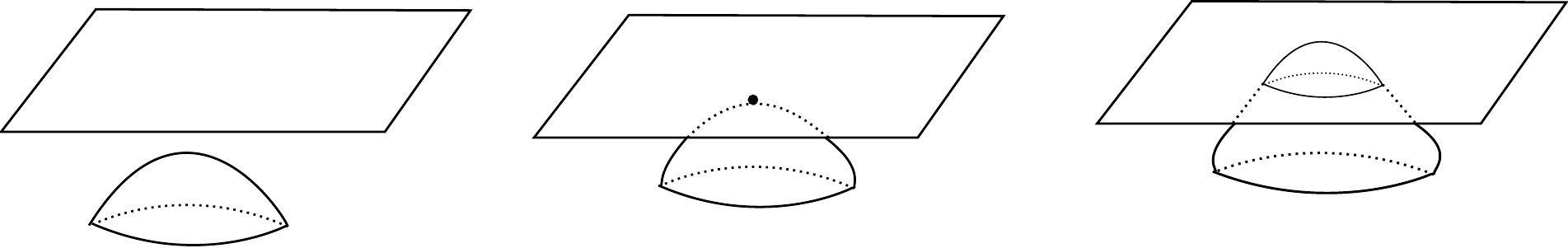}
\caption{A type $E$ singularity}
\label{esings}
\end{figure} 
\item Two complex conjugate tangent sheets intersect tangentially at a single point. The value of the invariant on the new circle is $0$.
\item Two sheets intersect as a plane passing through the surface defined by $z=x^2-y^2$. See Figure \ref{hsings}. If the two components of the self intersection are different before and after, it may have caused two different immersed circles to join together (or split apart). Globally, this changes the invariant by a multiple of $8$.
\begin{figure}
\includegraphics[scale=0.8]{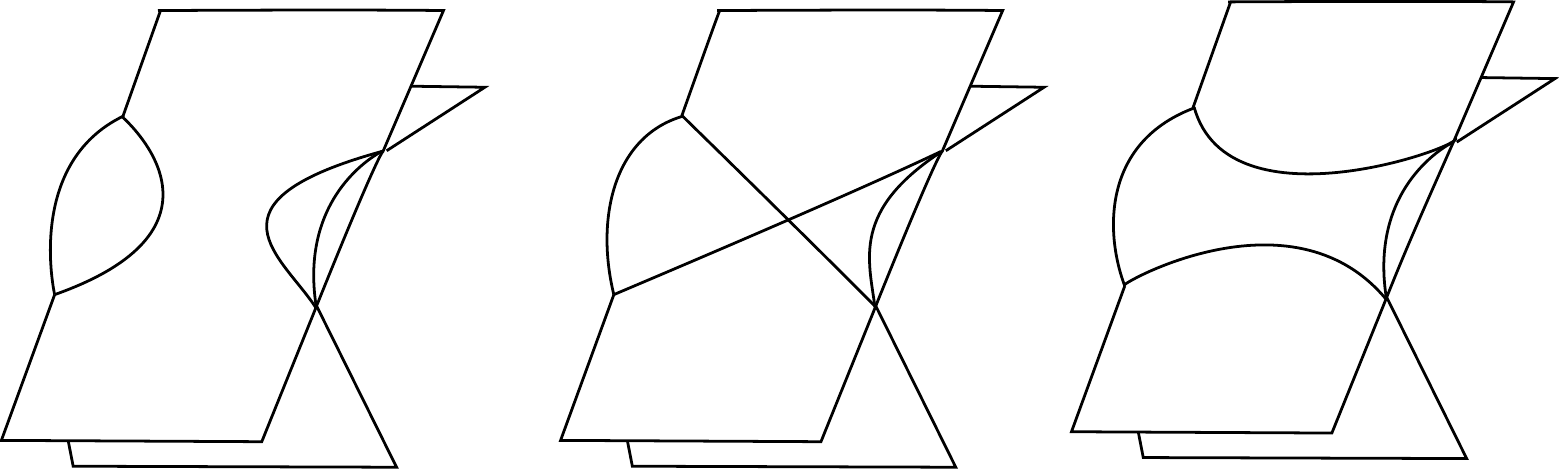}
\caption{A type $H$ singularity}
\label{hsings}
\end{figure} 
\item Two Whitney umbrellas collide and annihilate each other, either along a solitary self intersection or a real self intersection. See Figure \ref{bksings}. The invariant does not change.
\begin{figure}
\includegraphics[scale=0.8]{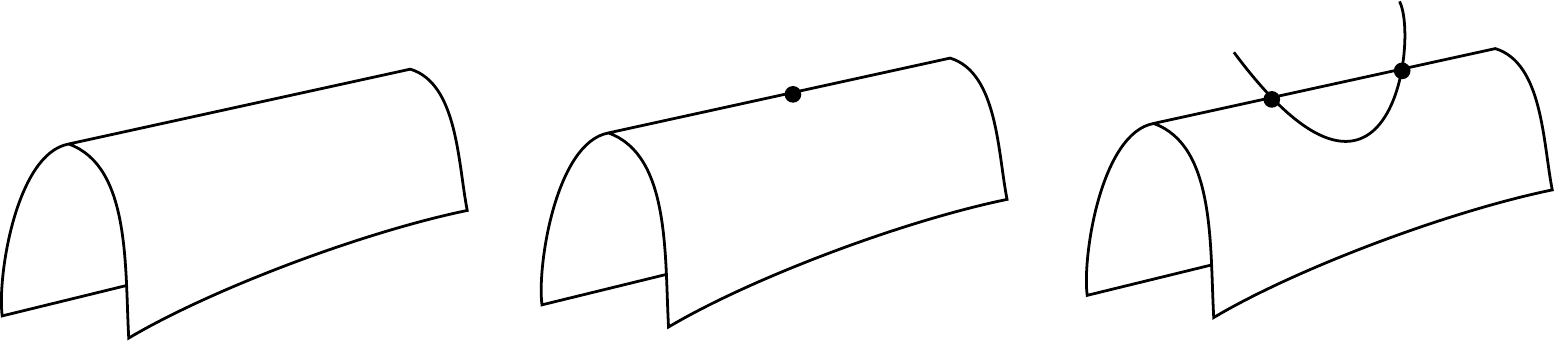}
\includegraphics[scale=0.8]{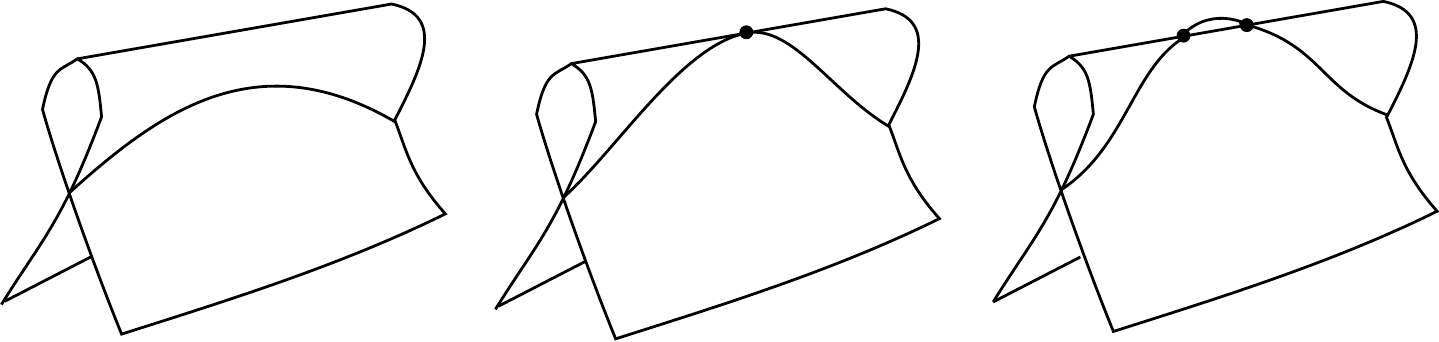}
\caption{A type $B$ singularity and a type $K$ singularity}
\label{bksings}
\end{figure}

\item A real plane passes through a Whitney umbrella. See Figure \ref{csings}. The invariant does not change.
\begin{figure}
\includegraphics[scale=0.8]{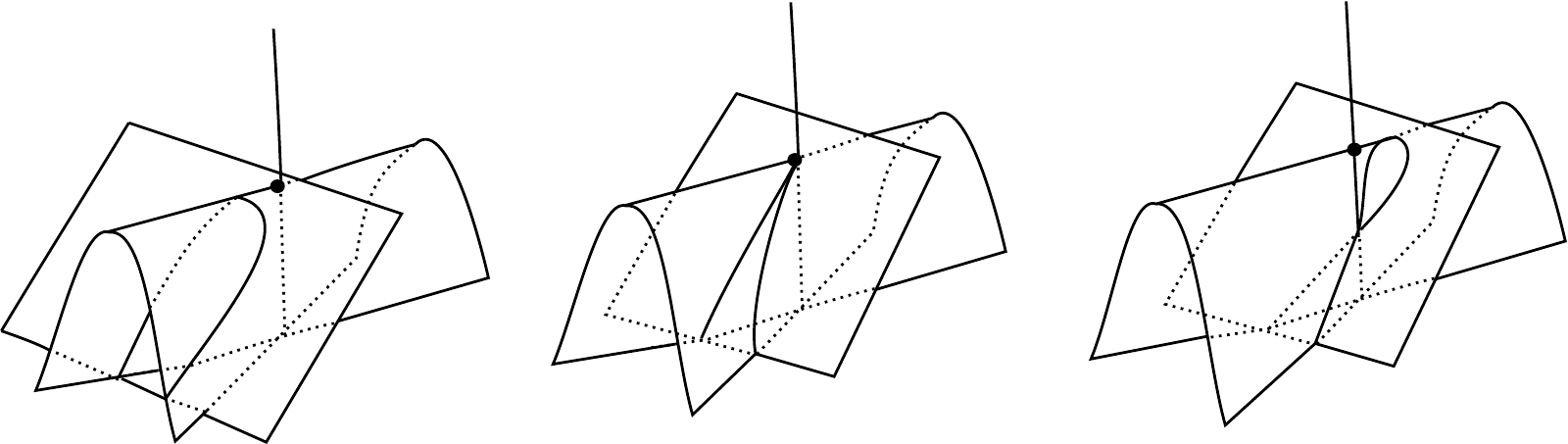}
\caption{A type $C$ singularity}
\label{csings}
\end{figure}

\end{itemize}
\end{proof}
\begin{remark}
Several of these co-dimension one passages have been studied for the smooth case by Goryunov \cite{GORY}, in our setting additional co-dimension one singularities arose from the additional structure from complex parts of the surface. We have followed his notation for the different kinds of singularities, using $'$ to denote variants arising from solitary/complex conjugate pieces.
\end{remark}

\section{Examples}\label{examples}
In this section we present some examples of surfaces together with the value of their fourfold pushoff invariant.
 
\begin{example}
The Roman surface is an example of a parametrized projective plane. It has 6 Whitney umbrellas and the fourfold pushoff takes value $0$. It is defined by the equation $x^2y^2+y^2z^2+z^2x^2=xyz$ in affine coordinates.
\end{example}

\begin{example}
The parameterization $x(t,s)=\frac{t^3s+ts^3}{t^4+s^4}, z(t,s)=\frac{2ts(t^2-s^2)}{t^4+s^4}$ parameterizes a curve looking like the symbol $\infty$. In affine coordinates we get a similar symbol by the solutions to the equation $y^4+y^2=x^2(1-x^2)$. By rotating it around a line in space we get a real surface parametrized by a torus such that its real self intersection is diffeomorphic to a circle (this surface will then be parametrized by degree 8). The pushoff of this circle will result in four circles, none of them linked to the self intersection so the invariant is zero.
\end{example}

\begin{example}
Take our earlier parametrization $x(s,t),z(s,t)$ of the $\infty$ symbol.  We want to rotate this parametrization by applying the following trick. We can parametrize a circle easily by using $p_1(u,v)=\frac{2uv}{u^2+v^2}, p_2(u,v)=\frac{u^2-v^2}{u^2+v^2}$. This allows us to consider $p_1,p_2$ as the $\sin$ and $\cos$ functions. They allow us to first apply a rotation matrix to the parametrization of the $\infty$ symbol (by just multiplying with the parametrization), then moving it to the side, and the rotating it around the $z-$axis as earlier. This will result in a degree $16$ surface which still has a circle as self intersection while the pushoff is linked, giving a value on the invariant of $\pm4$ depending on our choice of direction of rotation.
\end{example}

\begin{example}
The equation $x^4+y^4+(z^2+t^2)y^2-(z^2+t^2)x^2=0$ defines a surface having with two components in the pushoff of the self intersection. The value of the invariant is $0$.
\end{example}

\begin{example}
The equation $x^4+y^4+2(z^2+t^2)y^2+(z^2+t^2)x^2=0$ defines a surface with a real part consisting only of a solitary self intersection looking like the line defined by $x=y=0$. The value of the invariant is $0$.
\end{example}

\begin{example}
The equation $(t^2+z^2)y^2=x^2(x-z)(x-3z-t)$ defines a surface which has a self intersection located at the line defined by $x=y=0$. The self intersection consists of both solitary and non solitary parts. The value of the invariant is $0$.
\end{example}

\begin{example}
The equation $(t^2+z^2)y=\pm2tzx$ defines a surface which has a self intersection located at the line defined by $t=z=0$. The value of the invariant is $\pm1$.
\end{example}

\section*{acknowledgements}
I would like to thank Tobias Ekholm for the original idea and for many interesting discussions.


\begin{thebibliography}{MMMM}

\addtocounter{bibno}{1}
\bibitem{GORY}
V. Goryunov. Local invariants of mappings of surfaces into three-space,
`Arnold-Gelfand Mathematical Seminars. Geometry and Singularity Theory' (V.I. Arnold and I.M.Gelfand, eds.), Birkhäuser, 1996, 223-255. 
\addtocounter{bibno}{1}
\bibitem{KIME}
R.Kirby, P. Melvin. Local surgery formulas for quantum invariants and the Arf invariant,
Geom. Topol. Monogr. 7 (2004) 213-233
\addtocounter{bibno}{1}
\bibitem{VIRO}
O.Viro. Whitney Number of Closed Real Algebraic Affine Curve of Type I 
Moscow Mathematical Journal 6:1 (2006)
\addtocounter{bibno}{1}
\bibitem{KIRK}
C. A. Hobbs and N. P. Kirk, On the classification and bifurcation of multigerms of maps from surfaces to 3-space, Math. Scand., 89(2001), 57-96.

\bibitem{VIRO2}
O.Viro. Encomplexing the writhe,
[arXiv:math.GT/0005162v1]


\end{thebibliography}
\end{document}